\documentclass{amsart}

\usepackage{bbding}
\usepackage{amsmath}
\usepackage{amssymb,amsfonts,amsthm}
\usepackage[all]{xy}
 \usepackage{color}

\newtheorem{theorem}{Theorem}[section]
\newtheorem{lemma}[theorem]{Lemma}
\newtheorem{proposition}[theorem]{Proposition}

\theoremstyle{definition}
\newtheorem{definition}[theorem]{Definition}

\newtheorem{remark}[theorem]{Remark}
\numberwithin{equation}{section}

\newcommand{\blankbox}[2]

\begin{document}
\title{Non-weight modules over a Schr{\"o}dinger-Virasoro type algebra}
\author{Jiajia Wen}
\address{School of Mathematics, Hangzhou Normal University,
Hangzhou, 311121, China}
\email{2020111008010@stu.hznu.edu.cn}

\author{Zhongyin Xu}
\address{School of Mathematics, Hangzhou Normal University,
Hangzhou, 311121, China}
\email{2019210216022@stu.hznu.edu.cn}

\author{Yanyong Hong}
\address{School of Mathematics, Hangzhou Normal University,
Hangzhou, 311121, China}
\email{yyhong@hznu.edu.cn}
\thanks{This work was supported by the National Natural Science Foundation of China (No. 12171129, 11871421), the Zhejiang Provincial Natural Science Foundation of China (No. LY20A010022) and the Scientific Research Foundation of Hangzhou Normal University (No. 2019QDL012).}
\begin{abstract}
In this paper, we give a complete classification of all free $U(\mathbb{C}L_0 \oplus \mathbb{C}Y_0\oplus \mathbb{C}M_0)$-modules of rank 1 over a Schr{\"o}dinger-Virasoro type algebra $\mathfrak{tsv}$.
\end{abstract}
\maketitle
\section{Introduction}
Throughout this paper, we denote by $\mathbb{C}$, $\mathbb{C}^*$, $\mathbb{Z}$, $\mathbb{Z}^*$, $\mathbb{Z}_+$ and $\mathbb{N}$ the sets of all
complex numbers, nonzero complex numbers, integers, nonzero integers, non-negative integers and positive integers respectively.
All vector spaces (resp. algebras, modules) are considered to be over $\mathbb{C}$. For a non-zero polynomial $f(x_1, x_2, x_3)\in \mathbb{C}[x_1, x_2, x_3]$, we denote the degree of $f(x_1, x_2, x_3)$ in $x_i$ by $deg_{x_i}(f(x_1, x_2, x_3))$ for any $i\in \{1, 2, 3\}$.

In this paper, our studying object is a Schr{\"o}dinger-Virasoro type algebra $\mathfrak{tsv}$. $\mathfrak{tsv}$ is an infinite-dimensional Lie algebra with basis $\{L_m, Y_m, M_m|m\in \mathbb{Z}\}$ satisfying the following relations
\begin{align}
&[L_{m},L_{n}]=(m-n)L_{m+n}\label{ll},\\
&[L_{m},Y_{n}]=-(m+n)Y_{m+n}\label{ly},\\
&[L_{m},M_{n}]=-(3m+n)M_{m+n}\label{lm},\\
&[Y_{m},Y_{n}]=(m-n)M_{m+n}\label{yy},\\
&[Y_{m},M_{n}]=[M_{m},M_{n}]=0\label{00},
\end{align}
where $m$, $n\in\mathbb{Z}$. It is a twisted case of the deformative Schr{\"o}dinger-Virasoro algebras,
which was introduced in study of the Schr{\"o}dinger equation of free particles in non-equilibrium statistical physics (see \cite{H1, HU, RU}).  Central extensions,  non-degenerate symmetric invariant bilinear forms,  derivations and automorphisms, 2-cocycles and Harish-Chandra modules of general twisted deformative Schr{\"o}dinger-Virasoro algebras were investigated in \cite{RU, FLZ, WLX, L, Liu} respectively. There also have been some works only on $\mathfrak{tsv}$.   Lie bialgebra structures, modules of the intermediate series, the Lie conformal algebra and quantization of $\mathfrak{tsv}$ were studied in \cite{FLL, FDL, WXX, WX} respectively. From these results, it is found that $\mathfrak{tsv}$ has some particular properties which are different from many other cases of the twisted deformative Schr{\"o}dinger-Virasoro Lie algebras. Therefore, it is interesting to concentrate on investigating structure theory and representation theory of $\mathfrak{tsv}$.

Besides weight module theory of Lie algebras, non-weight module has drawn a lot of attention. Recently, an important class of non-weight modules which are called free $U(\mathfrak{h})$-modules were widely investiagted, where $U(\mathfrak{h})$ is the universal enveloping algebra
of the Cartan subalgebra $\mathfrak{h}$. The free $U(\mathfrak{h})$-modules over $sl_{n+1}$ were first
constructed in \cite{N1}, which were also introduced by a different method in \cite{TZ1}. Since then, there have been  many works on studying such modules over kinds of Lie algebras, for example,  finite-dimensional simple Lie algebras \cite{N2}, Kac-Moody algebras \cite{CTZ}, Virasoro algebra \cite{LZ, TZ}, Heisenberg-Virasoro algebra and $W(2,2)$ algebra \cite{CG}, $Vir(a, b)$ \cite{HCS}, Schr{\"o}dinger-Virasoro Lie algebra \cite{W, WZ}, Block algebra \cite{GWL, CY} and so on. Moreover, such modules over some Lie superalgebras such as  basic Lie superalgebras \cite{CZ}, super-Virasoro algebra \cite{YYX1}, untwisted $N = 2$ superconformal algebras \cite{YYX2} were also studied. In this paper, we plan to classify all free $U(\mathbb{C}L_0 \oplus \mathbb{C}Y_0\oplus \mathbb{C}M_0)$-modules of rank 1 over $\mathfrak{tsv}$. It should be pointed out that the free $U(\mathbb{C}L_0)$-modules of rank 1 over general twisted deformative Schr{\"o}dinger-Virasoro Lie algebras were characterized in \cite{CC}. Obviously, it is more difficult to consider free $U(\mathbb{C}L_0 \oplus \mathbb{C}Y_0\oplus \mathbb{C}M_0)$-modules of rank 1 than free $U(\mathbb{C}L_0)$-modules of rank 1. It can be seen from the computing complexity in this paper. Our main result is Theorem \ref{theorem of isomorphism}, which gives a complete classification of all modules over $\mathfrak{tsv}$  which are free $U({\mathbb{C}}L_0\oplus{\mathbb{C}}M_0\oplus{\mathbb{C}}Y_0)$-modules of rank 1.

This paper is organized as follows. In Section 2, we introduce a class of non-weight $\mathfrak{tsv}$-modules which are free $U({\mathbb{C}}L_0\oplus{\mathbb{C}}M_0\oplus{\mathbb{C}}Y_0)$-modules of rank 1 and investigate the module structures of these modules.
Section 3 is devoted to classifying all modules over $\mathfrak{tsv}$  which are free $U({\mathbb{C}}L_0\oplus{\mathbb{C}}M_0\oplus{\mathbb{C}}Y_0)$-modules of rank 1.

\section{ Preliminaries}

By the definition of $\mathfrak{tsv}$, we obtain some important formulas, which will be frequently used in the sequent.

\begin{proposition}\label{lie} In the universal enveloping algebra $U(\mathfrak{tsv})$ of $\mathfrak{tsv}$, the following formulas hold for any $i\in \mathbb{Z}_+$:
\begin{align}
&M_{m}L^i_{0}=(L_0+m)^iM_m\label{mli},\\
&M_{m}M^i_{0}=M^i_0M_m\label{mmi},\\
&M_{m}Y^{i}_{0}=Y^i_{0}M_m\label{myi},\\
&Y_{m}L^i_{0}=(L_0+m)^iY_m\label{yli},\\
&Y_{m}M^i_{0}=M^i_0Y_m\label{ymi},\\
&L_mL^i_0=(L_0+m)^iL_m\label{lli},\\
&Y_{m}Y^i_{0}=Y^i_{0}Y_m+miY^{i-1}_{0}M_{m}\label{yyi},\\
&L_mM^i_0=M^i_0L_m-3miM^{i-1}_{0}M_m\label{lmi},\\
&L_mY^i_0=Y^i_0L_m-miY^{i-1}_{0}Y_m-\frac{i(i-1)}{2}m^2Y^{i-2}_{0}M_m\label{lyi}.
\end{align}
\end{proposition}
\begin{proof}
Obviously, (\ref{mli})-(\ref{lyi}) hold when $i=0$. When $i\geq 1$,
we only check $\eqref{yyi}-\eqref{lyi}$, since others can be easily verified by induction.

For $\eqref{yyi}$, we prove it by induction.
For $i=1$, the formula $Y_mY_0=Y_0Y_m+mM_m$ is straightforward.
Assume that $\eqref{yyi}$ holds for $i=k-1$,
i.e. $Y_mY^{k-1}_0=Y^{k-1}_0Y_m+m(k-1)Y_0^{k-2}M_m$
holds. Then for $i=k$, we have

\begin{eqnarray*}
Y_mY^k_0&=&Y^{k-1}_0Y_mY_0+m(k-1)Y^{k-2}_0M_mY_0\\
&=&Y^{k-1}_0(Y_0Y_m+mM_m)+m(k-1)Y^{k-2}_0Y_0M_m\\
&=&Y^k_0Y_m+mkY^{k-1}_0M_m,
\end{eqnarray*}
which implies that $\eqref{yyi}$ holds.

For $\eqref{lmi}$, we also use induction to prove it. For $i=1$, $L_mM_0=M_0L_m-3mM_m$  obviously holds. Assume that  $\eqref{lmi}$ holds for $i=k-1$, i.e.
\begin{align*}
L_mM^{k-1}_0=M^{k-1}_0L_m-3m(k-1)M^{k-2}_0M_m
\end{align*}
holds. Thus for $i=k$, we have
$$
\begin{aligned}
L_mM^k_0&=M^{k-1}_0L_mM_0-3m(k-1)M^{k-2}_0M_m\\
&=M^{k-1}_0(M_0L_m-3mM_m)-3m(k-1)M^{k-2}_0M_0M_m\\
&=M^k_0L_m-3mkM^{k-1}_0M_m,
\end{aligned}
$$
which implies that $\eqref{lmi}$ holds.

For \eqref{lyi}, we still use induction to prove it. For $i=1$, $L_mY_0=Y_0L_m-mY_m$ is straightforward. Assume that \eqref{lyi} holds for $i=k-1$, i.e.
\begin{align*}
L_mY^{k-1}_0=Y^{k-1}_0L_m-m(k-1)Y^{k-2}_0Y_m-\frac{(k-1)(k-2)}{2}m^2Y^{k-3}_0M_m
\end{align*}
holds. Thus for $i=k$, we have
$$
\begin{aligned}
L_mY^k_0&=Y^{k-1}_0L_mY_0-m(k-1)Y^{k-2}_0Y_mY_0-\frac{(k-1)(k-2)}{2}m^2Y^{k-3}_0M_mY_0\\
&=Y^{k-1}_0(Y_0L_m-mY_m)-m(k-1)Y^{k-2}_0(Y_0Y_m+mM_m)\\
&\qquad -\frac{(k-1)(k-2)}{2}m^2Y^{k-3}_0Y_0M_m\\
&=Y^k_0L_m-mkY^{k-1}_0Y_m-\frac{k(k-1)}{2}m^2Y^{k-2}_0M_m,
\end{aligned}
$$
which implies that \eqref{lyi} holds.
\end{proof}

\begin{definition}
Let $\mathbb{C}\left[s,t,v\right]$ be the polynomial algebra in variables $s$, $t$ and $v$ with coefficients in $\mathbb{C}$. For $\lambda\in\mathbb{C}^*$, $a$, $b\in\mathbb{C}$,  $q\in \mathbb{N}$, $\tau_i$, $\gamma_j\in t\mathbb{C}[t]$, $\gamma_0=0$, $i\in \{0, 1, \cdots, q\}$ and $j\in \mathbb{Z}$ , define the action of $\mathfrak{tsv}$ on $\Phi(\lambda, a, b, q,\{\tau_i\}, \{\gamma_j\})=\mathbb{C}\left[s,t,v\right]$ as follows: for any $f(s,t,v)\in\mathbb{C}\left[s,t,v\right]$,
\begin{align}
	\begin{split}\label{dingyi1}
	L_m. f(s,t,v)&=\lambda^m\Bigg\{\bigg[s+\frac{m}{t}\Big(\sum\limits_{i=0}^{q}
	(-\tau_i+\frac{3}{i+1}\tau_i-\frac{3t}{i+1}\frac{\partial \tau_i}{\partial t})v^{i+1}-\frac{m}{2}t\sum\limits_{i=0}^{q}i\tau_iv^{i-1}\\
   &-\frac{m}{2}\left(\sum\limits_{i=0}^{q}\tau_iv^i\right)^2+3amtv-(am^2t+b)\sum\limits_{i=0}^{q}\tau_iv^i\Big)+\frac{\gamma_m}{t}\bigg]\\
   &\cdot f(s+m,t,v)-3mt\frac{\partial}{\partial t}f(s+m,t,v)-m\left(m\sum\limits_{i=0}^{q}\tau_iv^i+v+am^2t+b\right)\\
   &\cdot\frac{\partial}{\partial v}f(s+m,t,v)-\frac{m^2}{2}t\frac{\partial^2}{\partial v^2}f(s+m,t,v)\Bigg\},
   \end{split}
\end{align}
\begin{align}
\label{dingyi2}&M_m. f(s,t,v)=\lambda^mt f(s+m,t,v),\\
\label{dingyi3}&Y_m. f(s,t,v)=\lambda^m \Bigg[\bigg(m\sum\limits_{i=0}^{q}\tau_iv^i+v+am^2t+(1-\delta_ {m,0})b\bigg)f(s+m,t,v)\\
&+ mt\frac{\partial}{\partial v}f(s+m,t,v)\Bigg],\nonumber
\end{align}
where $m\in\mathbb{Z}$.
\end{definition}

\begin{proposition}\label{lie}
All $\Phi(\lambda, a, b, q,\{\tau_i\}, \{\gamma_j\})$ are $\mathfrak{tsv}$-modules under the actions $\eqref{dingyi1}-\eqref{dingyi3}$.
\end{proposition}
\begin{proof}
For $m$, $n\in\mathbb{Z}$, according to (\ref{dingyi1})-(\ref{dingyi3}), we have
{\small $$
\begin{aligned}
&L_m. Y_n.f(s,t,v)=\lambda^{m+n}\Bigg\{\bigg[(s+D_m)\Big(n\sum\limits_{i=0}^{q}
\tau_iv^i+v+an^2t+(1-\delta_ {n,0})b\Big)-mn\Big(\sum\limits_{i=0}^{q}i\tau_iv^{i-1}\Big)\\
&\cdot (m\sum\limits_{i=0}^{q}\tau_iv^i+v+am^2t+b)
-\frac{m^2n}{2}t\sum\limits_{i=0}^{q}i(i-1)\tau_iv^{i-2}-3mnt
\sum\limits_{i=0}^{q}\frac{\partial\tau_i}{\partial t}v^i\\
&-m\Big(m\sum\limits_{i=0}^{q}\tau_iv^i+v+am^2t+b\Big)\bigg]f(s+m+n,t,v)+\bigg[-m\Big(n\sum\limits_{i=0}^{q}\tau_iv^i+v+an^2t+(1-\delta_ {n,0})b\Big)\\
&\cdot\Big(m\sum\limits_{i=0}^{q}\tau_iv^i+v+am^2t+b\Big)-m^2nt\sum\limits_{i=0}^{q}i\tau_iv^{i-1}+nt(s+D_m)-3mnt-m^2t\bigg]
\\
&\cdot\frac{\partial}{\partial v}f(s+m+n,t,v)+\bigg[-\frac{3}{2}m^2nt\sum\limits_{i=0}^{q}\tau_iv^i-\Big(mn+\frac{m^2}{2}\Big)tv-\frac{1}{2}am^2n^2t^2-am^3nt^2\bigg]
\\
&\cdot \frac{\partial^2}{\partial v^2}f(s+m+n,t,v)
+\Big(-\frac{m^2n}{2}t^2\Big)\frac{\partial^3}{\partial v^3}f(s+m+n,t,v)\\
&+\bigg[-3mt\Big(n\sum\limits_{i=0}^{q}\tau_i v^i+v+an^2t+(1-\delta_ {n,0})b\Big)\bigg]\frac{\partial}{\partial t}f(s+m+n,t,v)\\
&+(-3mnt^2)\frac{\partial^2}{\partial t\partial v}f(s+m+n,t,v)\Bigg\},
\end{aligned}$$}

{\small $$
\begin{aligned}
&Y_n. L_m. f(s,t,v)=\lambda^{m+n}\Bigg\{\bigg[(s+n+D_m)\Big(n\sum\limits_{i=0}^{q}\tau_iv^i+v+an^2t+(1-\delta_ {n,0})b\Big)+nt\frac
{\partial D_m}{\partial v}\bigg]\\
&f(s+m+n,t,v)+\bigg[-m\Big(m\sum\limits_{i=0}^{q}\tau_iv^i+v+am^2t+b\Big)\Big(n\sum\limits_{i=0}^{q}\tau_iv^i+v+an^2t+(1-\delta_ {n,0})b\Big)\\
&-mnt\Big(m\sum\limits_{i=0}^{q}i\tau_iv^{i-1}+1\Big)+nt(s+n+D_m)\bigg]\frac{\partial}{\partial v}f(s+m+n,t,v)
\end{aligned}
$$}
{\small $$
\begin{aligned}
&+\bigg[-\frac{3}{2}m^2nt\sum\limits_{i=0}^{q}\tau_iv^i-\Big(mn+\frac{m^2}{2}\Big)tv-\frac{1}{2}am^2n^2t^2-am^3nt^2\bigg]
\frac{\partial^2}{\partial v^2}f(s+m+n,t,v)\\
&+\Big(-\frac{m^2n}{2}t^2\Big)\frac{\partial^3}{\partial v^3}f(s+m+n,t,v)+\bigg[-3mt\Big(n\sum\limits_{i=0}^{q}\tau_i v^i+v+an^2t+(1-\delta_ {n,0})b\Big)\bigg]\\
&\cdot \frac{\partial}{\partial t}f(s+m+n,t,v)+(-3mnt^2)\frac{\partial^2}{\partial t\partial v}f(s+m+n,t,v) \Bigg\},
\end{aligned}
$$}
where
{\small\begin{align*}
&D_m=\frac{m}{t}\Bigg(\sum\limits_{i=0}^{q}\bigg(-\tau_i+\frac{3}{i+1}\tau_i
-\frac{3t}{i+1}\frac{\partial\tau_i}{\partial t}\bigg)v^{i+1}-\frac{m}{2}t\sum\limits_{i=0}^{q}i\tau_iv^{i-1}\\
&\quad\quad\quad\quad-\frac{m}{2}\bigg(\sum\limits_{i=0}^{q}
\tau_iv^i\bigg)^2+3amtv-(am^2t+b)\sum\limits_{i=0}^{q}\tau_iv^i\Bigg)+\frac{\gamma_m}{t}.
\end{align*}}
Thus, we obtain
{\small \begin{eqnarray*}
&&L_m.Y_n.f(s,t,v)-Y_n. L_m. f(s,t,v)\\
&=&\lambda^{m+n}\bigg\{\Big[-(m^2+n^2+2mn)\sum\limits_{i=0}^{q}\tau_iv^i-(m+n)v-a(m^3+n^3+3m^2n+3mn^2)t\\
&&-(m+(1-\delta_ {n,0})n)b\Big]f(s+m+n)+\Big[-(m^2+n^2+2mn)t\Big]\frac{\partial}{\partial v}f(s+m+n,t,v)\bigg\}\\
&=&\lambda^{m+n}\bigg\{\Big[-(m^2+n^2+2mn)\sum\limits_{i=0}^{q}\tau_iv^i-(m+n)v-a(m^3+n^3+3m^2n+3mn^2)t\\
&&-(m+n)(1-\delta_ {m+n,0})b\Big]f(s+m+n)+\Big[-(m^2+n^2+2mn)t\Big]\frac{\partial}{\partial v}f(s+m+n,t,v)\bigg\}\\
&=&-(m+n)Y_{m+n}.f(s,t,v)\\
&=&[L_m,Y_n]. f(s,t,v).
\end{eqnarray*}}
Other equalities can be checked similarly.
\end{proof}
\begin{remark}
(1) $\Phi(\lambda, a, b, q,\{\tau_i\}, \{\gamma_j\})$ is a free $U(\mathbb{C}L_0\bigoplus\mathbb{C}M_0\bigoplus\mathbb{C}Y_0)$-module of rank 1, since
$L_0. f(s,t, v)=sf(s,t,v)$, $M_0. f(s,t,v)=tf(s,t,v)$ and $Y_0. f(s,t,v)=vf(s,t,v)$ for any $f(s,t,v)\in \mathbb{C}[s, t, v]$.

(2) $\Phi(\lambda, a, b, q,\{\tau_i\}, \{\gamma_j\})$ is reducible as a module of $\mathfrak{tsv}$. In fact, it is easy to see that $t^k\mathbb{C}[s,t,v]$ is a submodule of $\Phi(\lambda, a, b, q,\{\tau_i\}, \{\gamma_j\})$ for any $k\in \mathbb{Z}_+$.
\end{remark}
\begin{proposition}
The quotient module $t^k\mathbb{C}[s,t,v]/t^{k+1}\mathbb{C}[s,t,v]$ is reducible as a module of $\mathfrak{tsv}$ for each $k\in \mathbb{Z}_+$.
\end{proposition}
\begin{proof}
Obviously, $t^k\mathbb{C}[s,t,v]/t^{k+1}\mathbb{C}[s,t,v]\cong t^k\mathbb{C}[s,v]$ as vector spaces for each $k\in \mathbb{Z}_+$.
Then the module action of $\mathfrak{tsv}$ on $t^k\mathbb{C}[s,v]$ as the natural action on the quotient module $t^k\mathbb{C}[s,t,v]/t^{k+1}\mathbb{C}[s,t,v]$ is given as follows: for any $f(s,v)\in\mathbb{C}[s,v]$,
$$
\begin{aligned}
L_m. t^k f(s,v)&=\lambda^m\Bigg\{\bigg[s+m\Big(-\sum\limits_{i=0}^{q}\tau_{i,1}(v+b)v^{i}+3amv \Big)+\gamma_{m,1}\bigg]t^kf(s+m,v)\\
&-3mkt^kf(s+m,v)-m(v+b)t^k\frac{\partial}{\partial v}f(s+m,v)\Bigg\},\\
M_m. t^kf(s,v)&=0,\\
Y_m. t^kf(s,v)&=\lambda^m (v+(1-\delta_{m, 0})b)t^kf(s+m,v),
\end{aligned}
$$
where $m\in\mathbb{Z}$, $\gamma_{m,1}$ is the coefficient of $t$ in $\gamma_m$ and $\tau_{i,1}$ is the coefficient of $t$ in $\tau_i$ for each $i\in \{0, 1, \cdots, q\}$.

It is easy to see that $ t^k(v+b)^j\mathbb{C}[s, v]$ is a submodule of $t^k\mathbb{C}[s,v]$ for each $j\in \mathbb{Z}_+$. Then the quotient module $t^k\mathbb{C}[s,t,v]/t^{k+1}\mathbb{C}[s,t,v]$ is reducible.
\end{proof}

\begin{remark}
In fact, for each $j\in \mathbb{Z}_+$, the module of $\mathfrak{tsv}$ $t^k\mathbb{C}[s,v]$ given in the proof above also has the following submodule filtration
\begin{eqnarray*}
\cdots \subseteq t^k(v+b)^{j+1} \mathbb{C}[s,v]\subseteq t^k(v+b)^{j} \mathbb{C}[s,v]\subseteq \cdots \subseteq t^k\mathbb{C}[s,v].
\end{eqnarray*}

\end{remark}

\begin{proposition}
$\Phi(\lambda, a, b, q,\{\tau_i\}, \{\gamma_j\})$ and $\Phi(\lambda', a', b', q', \{\tau_i'\}, \{\gamma'_j\} )$ are isomorphic if and only if $(\lambda, a, b, q,\{\tau_i\}, \{\gamma_j\})=(\lambda', a', b', q', \{\tau_i'\}, \{\gamma'_j\})$.
\end{proposition}
\begin{proof}
Let $\varphi: \Phi(\lambda, a, b, q,\{\tau_i\}, \{\gamma_j\}) \longrightarrow \Phi(\lambda', a', b', q', \{\tau_i'\}, \{\gamma'_j\} )$
be an isomorphism of $\mathfrak{tsv}$-modules. Since
\begin{align*}
\varphi(L_0^iM_0^jY_0^k. 1) =L_0^iM_0^jY_0^k.\varphi(1),
\end{align*}
we have $\varphi(s^it^jv^k)=s^it^jv^k\varphi(1)$. Since $\varphi$ is an isomorphism of vector spaces, $\varphi(1)\in\mathbb{C}^*$. We have
$$
\begin{aligned}
&\varphi(Y_m.1)=\varphi\left(\lambda^m \bigg(m\sum\limits_{i=0}^{q}\tau_iv^i+v+am^2t+(1-\delta_{m,0})b\bigg)\right)\\
&\qquad\qquad=\lambda^m\left(m\sum\limits_{i=0}^{q}\tau_iv^i+v+am^2t+(1-\delta_{m,0})b\right)\varphi(1)\\
&\qquad\qquad=Y_m.\varphi(1)=(\lambda')^m\left(m\sum\limits_{i=0}^{q}(\tau'_i)v^i+v+a'm^2t+(1-\delta_{m,0})b'\right)\varphi(1).
\end{aligned}
$$
Therefore, $\lambda=\lambda'$,  $a= a'$, $b=b'$, $q=q'$ and $\tau_i=\tau_i'$ for each $i$.
Similarly,
{\small $$
\begin{aligned}
&\varphi(L_m.1)\\
&=\varphi \Bigg\{\lambda^m\bigg[s+\frac{m}{t}\Big(\sum\limits_{i=0}^{q}
(-\tau_i+\frac{3}{i+1}\tau_i-\frac{3t}{i+1}\frac{\partial \tau_i}{\partial t})v^{i+1}-\frac{m}{2}t\sum\limits_{i=0}^{q}i\tau_iv^{i-1}\\
&-\frac{m}{2}\left(\sum\limits_{i=0}^{q}\tau_iv^i\right)^2
+3amtv-(am^2t+b)\sum\limits_{i=0}^{q}\tau_iv^i\Big)+\frac{\gamma_m}{t}\bigg]\Bigg\}\\
&=\lambda^m\bigg[s+\frac{m}{t}\Big(\sum\limits_{i=0}^{q}
(-\tau_i+\frac{3}{i+1}\tau_i-\frac{3t}{i+1}\frac{\partial \tau_i}{\partial t})v^{i+1}-\frac{m}{2}t\sum\limits_{i=0}^{q}i\tau_iv^{i-1}\\
&-\frac{m}{2}\left(\sum\limits_{i=0}^{q}\tau_iv^i\right)^2+3amtv-(am^2t+b)\sum\limits_{i=0}^{q}\tau_iv^i\Big)+\frac{\gamma_m}{t}\bigg]\varphi(1)\\
&=L_m\cdot\varphi(1)\\
&=(\lambda')^m\bigg[s+\frac{m}{t}\Big(\sum\limits_{i=0}^{q}
(-\tau'_i+\frac{3}{i+1}\tau'_i-\frac{3t}{i+1}\frac{\partial \tau'_i}{\partial t})v^{i+1}-\frac{m}{2}t\sum\limits_{i=0}^{q}i\tau'_iv^{i-1}\\
&-\frac{m}{2}\left(\sum\limits_{i=0}^{q}\tau'_iv^i\right)^2+3a'mtv-(a'm^2t+b')\sum\limits_{i=0}^{q}\tau'_iv^i\Big)+\frac{\gamma'_m}{t}\bigg]\varphi(1).
\end{aligned}
$$}
Then we have $\gamma_j=\gamma_j'$ for any $j\in \mathbb{Z}$.

Then this proposition holds.
\end{proof}
\section{Main result}
In this section, we will classify all $\mathfrak{tsv}$-modules which are free $U({\mathbb{C}}L_0\oplus{\mathbb{C}}M_0\oplus{\mathbb{C}}Y_0)$-modules of rank 1.

\begin{theorem}\label{theorem of isomorphism}
Suppose that $B$ is a $\mathfrak{tsv}$-module which is a free $U({\mathbb{C}}L_0\oplus{\mathbb{C}}M_0\oplus{\mathbb{C}}Y_0)$-module of rank 1. Then $B \cong\Phi(\lambda, a, b, q, \{\tau_i\}, \{\gamma_j\})$, for some $a$, $b\in\mathbb{C}$, $\lambda\in\mathbb{C}^*$, $q\in \mathbb{N}$, $\tau_i$, $\gamma_j\in t\mathbb{C}[t]$, $i\in \{0, 1, \cdots, q\}$ and $j\in\mathbb{Z}$.
\end{theorem}

Before proving this theorem we have to present several important lemmas.

Note that $[L_0, Y_0]=0$, $[L_0, M_0]=0$ and $[Y_0, M_0]=0$. Therefore, $U({\mathbb{C}}L_0\oplus{\mathbb{C}}M_0\oplus{\mathbb{C}}Y_0)$ is just the polynomial algebra $\mathbb{C}[L_0, M_0, Y_0]$. Suppose that $B$ is a free  $U({\mathbb{C}}L_0\oplus{\mathbb{C}}M_0\oplus{\mathbb{C}}Y_0)$-module of rank $1$. Then we set $B=\mathbb{C}[L_0, M_0, Y_0]$, and $L_0. 1=L_0$, $Y_0. 1=Y_0$, $M_0. 1=M_0$, where $``."$ is just the module action of $U({\mathbb{C}}L_0\oplus{\mathbb{C}}M_0\oplus{\mathbb{C}}Y_0)$ on $B$. Moreover, we make a convention: for a term $L_0^iM_0^jY_0^k$ in $\mathbb{C}[L_0, M_0, Y_0]$, if $i<0$ or $j<0$ or $k<0$, $L_0^iM_0^jY_0^k=0$.
\begin{lemma}\label{lemma3.1}
Let $B$ be a free  $U({\mathbb{C}}L_0\oplus{\mathbb{C}}M_0\oplus{\mathbb{C}}Y_0)$-module of rank 1 over $\mathfrak{tsv}$. Assume that
\begin{align*}
&L_m. 1=g_m(L_0,M_0,Y_0),\\
&M_m. 1=a_m(L_0,M_0,Y_0),\\
&Y_m. 1=p_m(L_0,M_0,Y_0),
\end{align*}
where $ m\in \mathbb{Z}$, $g_m(L_0,M_0,Y_0)$,
$a_m(L_0,M_0,Y_0)$, $p_m(L_0,M_0,Y_0)\in B$, and $``."$ is the module action of $\mathfrak{tsv}$ on $B$. Then $g_m(L_0,M_0,Y_0)$,  $a_m(L_0,M_0,Y_0)$, $p_m(L_0,M_0,Y_0)$ determine the action of $L_m$, $M_m$ and $Y_m$ on $B$ completely.
\end{lemma}
\begin{proof}
Let $u(L_0,M_0,Y_0)=\Sigma_{i,j,k\geq0} a_{i,j,k}L^i_0M^j_0Y^k_0\in B$. Then by using $\eqref{mli}-\eqref{lyi}$, we have
{\small \begin{eqnarray*}
(A1)&& L_m.  u(L_0,M_0,Y_0)=L_m. \Sigma_{i,j,k\geq0} a_{i,j,k}L^i_0M^j_0Y^k_0 \\
&&=\Sigma_{i,j,k\geq0} a_{i,j,k}(m+L_0)^iL_m. M^j_0Y^k_0 \\\
&&=\Sigma_{i,j,k\geq0} a_{i,j,k}(m+L_0)^i(-3mjM^{j-1}_0M_m+M^j_0L_m).Y^k_0\\\
&&=\Sigma_{i,j,k\geq0}
a_{i,j,k}(m+L_0)^i\big(-3mjM^{j-1}_0Y^k_0a_m(L_0,M_0,Y_0)\\
&&-kmM^j_0Y^{k-1}_0p_m(L_0,M_0,Y_0)+M^j_0Y^k_0g_m(L_0,M_0,Y_0)\\
&&-m^2M^j_0\frac{k(k-1)}{2}Y^{k-2}_0a_m(L_0,M_0,Y_0)\big)\\\
&&=-3m(\frac{\partial}{\partial M_0}u(L_0+m,M_0,Y_0))a_m(L_0,M_0,Y_0)+u(L_0+m,M_0,Y_0)g_m(L_0,M_0,Y_0)\\
&&-m(\frac{\partial}{\partial Y_0}u(L_0+m,M_0,Y_0))p_m(L_0,M_0,Y_0)-\frac{m^2}{2}\frac{\partial^2}{\partial Y_0^2}u(L_0+m,M_0,Y_0))a_m(L_0,M_0,Y_0),
\end{eqnarray*}}

{\small\begin{eqnarray*}
(A2) &&M_m. u(L_0,M_0,Y_0)=M_m. \Sigma_{i,j,k\geq0} a_{i,j,k}L^i_0M^j_0Y^k_0\\
&&=\Sigma_{i,j,k\geq0} a_{i,j,k}(m+L_0)^iM_m.M^j_0Y^k_0=\Sigma_{i,j,k\geq0} a_{i,j,k}(m+L_0)^iM^j_0Y^k_0a_m(L_0,M_0,Y_0)\\
&&=u(L_0+m,M_0,Y_0)a_m(L_0,M_0,Y_0),
\end{eqnarray*}}

{\small\begin{eqnarray*}
(A3) &&Y_m\cdot u(L_0,M_0,Y_0)=Y_m\cdot \Sigma_{i,j,k\geq0} a_{i,j,k}L^i_0M^j_0Y^k_0\\
&&=\Sigma_{i,j,k\geq0} a_{i,j,k}(L_0+m)^iM^j_0Y_mY^k_0\\
&&=\Sigma_{i,j,k\geq0} a_{i,j,k}(L_0+m)^iM^j_0(Y_0^kY_m+mkY_0^{k-1}M_m)\\
&&=m\frac{\partial}{\partial Y_0}u(L_0+m,M_0,Y_0)a_m(L_0,M_0,Y_0)+u(L_0+m,M_0,Y_0)p_m(L_0,M_0,Y_0).
\end{eqnarray*}}
Therefore, by $(A1)$-$(A3)$, the module action of $L_m$, $M_m$ and $Y_m$ on $B$ is determined by $g_m(L_0,M_0,Y_0)$,  $a_m(L_0,M_0,Y_0)$, $p_m(L_0,M_0,Y_0)$ completely.
\end{proof}
\begin{lemma}\label{Lemma1}
$a_m(L_0,M_0,Y_0)\neq 0$ for any $m\in \mathbb{Z}$.
\end{lemma}
\begin{proof}
Note that $M_0. 1=M_0=a_0(L_0,M_0, Y_0)\neq 0$.
If there  exists some $m_0\in \mathbb{Z}^\ast$ such that $a_{m_0}(L_0,M_0,Y_0)=0$, by $(A2)$, we have $M_{m_0}. u(L_0,M_0,Y_0)=0$ for any $u(L_0,M_0,Y_0)\in T$.
Then using
$$
\begin{aligned}
&~~~~\left[L_{-m_0},M_{m_0}\right]. 1=2m_0M_0. 1\\
&=L_{-m_0}M_{m_0}. 1-M_{m_0}L_{-m_0}. 1\\
&=0-M_{m_0}. g_{-m_0}(L_0,M_0,Y_0)=0,\\
\end{aligned}
$$
we have $2m_0M_0. 1=0$. Since $m_0\neq 0$, $M_0.1=0$, which is a contradiction. Therefore, $a_m(L_0,M_0,Y_0)\neq 0$ for any $m\in \mathbb{Z}$.
\end{proof}
\begin{lemma}\label{Lemma2}
$a_m(L_0,M_0,Y_0)\in\mathbb{C}\left[M_0,Y_0\right]\setminus \{0\}$ for any $m\in \mathbb{Z}$.
\end{lemma}
\begin{proof}
By Lemma \ref{Lemma1}, assume that $a_m(L_0,M_0,Y_0)=\sum\limits_{i=0}^{k_m}b_{m,i}L^i_0 \neq 0$ for any $m\in \mathbb{Z}$, where $b_{m,i}\in \mathbb{C}\left[M_0,Y_0\right],\ k_m\in\mathbb{Z}_+,\ b_{m,k_m}\neq0$. If there exists $m_0\in \mathbb{Z}^\ast$ such that $deg_{L_0}(a_{m_0}(L_0,M_0,Y_0))\geq 1$,
by $(A2)$, we have
{\small
\begin{eqnarray*}
&&\left[M_{m_0},M_n\right]. 1
=M_{m_0}.(M_n. 1)-M_n.(M_{m_0}. 1)\\
&&=M_{m_0}. a_n(L_0,M_0,Y_0)-M_n. a_{m_0}(L_0,M_0,Y_0)\\
&&=\sum\limits_{i=0}^{k_n}b_{n,i}(L_0+m_0)^i\sum\limits_{j=0}^{k_{m_0}}b_{m_0,j}L^j_0-\sum\limits_{i=0}^{k_{m_0}}b_{m_0,i}(L_0+n)^i\sum\limits_{j=0}^{k_n}b_{n,j}L^j_0\\
&&=b_{n,k_n}b_{m_0,k_{m_0}}(m_0k_n-nk_{m_0})L^{k_{m_0}+k_n-1}_0\qquad\left(mod\sum\limits_{i=0}^{k_n+k_{m_0}-2}\mathbb{C}\left[Y_0,M_0\right]L^i_0\right),
\end{eqnarray*}}
for any $n\in \mathbb{Z}$. Since $\left[M_{m_0},M_n\right]. 1=0$, $(m_0k_n-nk_{m_0})=0$ for any $n\in \mathbb{Z}$.
Suppose that $m_0>0$. Let $n<0$. Since $k_{m_0}>0$ and $k_n\geq 0$, $m_0k_n-nk_{m_0}> 0$. We get a contradiction. Similarly, if $m_0<0$, letting $n>0$, we get $m_0k_n-nk_{m_0}<0$, a contradiction. Therefore, $a_m(L_0,M_0,Y_0)\in\mathbb{C}\left[M_0,Y_0\right]\setminus \{0\}$ for any $m\in \mathbb{Z}$.
\end{proof}
\begin{lemma}\label{Lemma3}
$a_m(L_0,M_0,Y_0)\in\mathbb{C}\left[M_0\right]\setminus \{0\}$, $\ p_m(L_0,M_0,Y_0)\in\mathbb{C}\left[M_0,Y_0\right]$ for any $m\in \mathbb{Z}$.
\end{lemma}
\begin{proof}
By Lemma \ref{Lemma2}, we assume that $a_m(L_0,M_0,Y_0)=\sum\limits_{i=0}^{l_m}c_{m,i}Y^i_0\neq 0$, where $c_{m,i}\in \mathbb{C}\left[M_0\right]$ for any $m\in \mathbb{Z}$ and $c_{m, l_m}\neq 0$.
Moreover, assume that $p_m(L_0,M_0,Y_0)$ $=\sum\limits_{i=0}^{h_m}d_{m,i}L^i_0$,
where $d_{m,i}\in \mathbb{C}\left[M_0,Y_0\right]$ for any $m\in \mathbb{Z}$.

By $(A2)$ and $(A3)$, we have
\begin{eqnarray}\label{[y,m]}
0&=&\quad\left[Y_m,M_n\right]. 1\\
&=&Y_m.(M_n. 1)-M_n. (Y_m. 1)\nonumber\\
&=&Y_m.\sum\limits_{i=0}^{l_n}c_{n,i}Y^i_0-M_n. \sum\limits_{i=0}^{h_m}d_{m,i}L^i_0 \nonumber\\
&=&\sum\limits_{i=0}^{l_n}\sum\limits_{j=0}^{h_m}c_{n,i}d_{m,j}Y^i_0L^j_0-\sum\limits_{i=0}^{l_n}\sum\limits_{j=0}^{h_m}c_{n,i}d_{m,j}(L_0+n)^jY^i_0\nonumber\\
&&+m\sum\limits_{i=0}^{l_n}\sum\limits_{j=0}^{l_m}ic_{n,i}c_{m,j}Y^{i+j-1}_0.\nonumber
\end{eqnarray}

If $\sum\limits_{i=0}^{l_n}\sum\limits_{j=0}^{h_m}c_{n,i}d_{m,j}Y^i_0L^j_0-\sum\limits_{i=0}^{l_n}\sum\limits_{j=0}^{h_m}c_{n,i}d_{m,j}(L_0+n)^jY^i_0 =0$, $deg_{L_0}(p_m(L_0, M_0, Y_0))=h_m=0$ or $p_m(L_0, M_0, Y_0)=0$, which mean that $p_m(L_0, M_0, Y_0)\in \mathbb{C}[M_0, Y_0]$.

If $\sum\limits_{i=0}^{l_n}\sum\limits_{j=0}^{h_m}c_{n,i}d_{m,j}Y^i_0L^j_0-\sum\limits_{i=0}^{l_n}\sum\limits_{j=0}^{h_m}c_{n,i}d_{m,j}(L_0+n)^jY^i_0 \neq 0$, we have
\begin{align*}
deg_{L_0}\Big(\sum\limits_{i=0}^{l_n}\sum\limits_{j=0}^{h_m}c_{n,i}d_{m,j}Y^i_0L^j_0-\sum\limits_{i=0}^{l_n}\sum\limits_{j=0}^{h_m}c_{n,i}d_{m,j}(L_0+n)^jY^i_0\Big)=0.
\end{align*}
Thus we get $deg_{L_0}(p_m(L_0, M_0, Y_0)=h_m\leq1$.

Suppose that there exists $m_0\in \mathbb{Z}^\ast$ such that $h_{m_0}=1$. Then
\begin{eqnarray}\label{eq2}
\eqref{[y,m]}=m_0\sum\limits_{i=0}^{l_n}\sum\limits_{j=0}^{l_{m_0}}ic_{n,i}c_{m_0,j}Y^{i+j-1}_0-n\sum\limits_{i=0}^{l_n}c_{n,i}d_{m_0,1}Y^i_0.
\end{eqnarray}

Letting $n=m_0$ in (\ref{eq2}),
we have
\begin{eqnarray*}
0&=&m_0(\sum\limits_{i=0}^{l_{m_0}}\sum\limits_{j=0}^{l_{m_0}}ic_{m_0,i}c_{m_0,j}Y^{i+j-1}_0-\sum\limits_{i=0}^{l_{m_0}}c_{m_0,i}d_{m_0,1}Y^i_0)\\
&=& m_0(\sum\limits_{j=0}^{l_{m_0}}\sum\limits_{i=0}^{l_{m_0}}jc_{m_0,j}c_{m_0,i}Y^{i+j-1}_0-\sum\limits_{i=0}^{l_{m_0}}c_{m_0,i}d_{m_0,1}Y^i_0)\\
&=& m_0 (\sum\limits_{i=0}^{l_{m_0}}c_{m_0,i}Y_0^i)\cdot(\sum\limits_{j=0}^{l_{m_0}}jc_{m_0,j}Y_0^{j-1}-d_{m_0,1}).
\end{eqnarray*}
Since $m_0\neq 0$ and $\sum\limits_{i=0}^{l_{m_0}}c_{m_0,i}Y_0^i\neq 0$, we get $d_{m_0,1}=\sum\limits_{j=0}^{l_{m_0}}jc_{m_0,j}Y_0^{j-1}$.

Letting $n=-m_0$ in (\ref{eq2}),
\begin{eqnarray*}
0&=& m_0\sum\limits_{i=0}^{l_{-m_0}}c_{-m_0,i}\Big(\sum\limits_{j=0}^{l_{m_0}}ic_{m_0,j}Y^{j}_0+d_{m_0,1}Y_0\Big)Y_0^{i-1},\\
&=&m_0\sum\limits_{i=0}^{l_{-m_0}}c_{-m_0,i}\Big(\sum\limits_{j=0}^{l_{m_0}}c_{m_0,j}(i+j)Y^{j}_0\Big)Y_0^{i-1}\\
&=&m_0\sum\limits_{i=0}^{l_{-m_0}}\sum\limits_{j=0}^{l_{m_0}}c_{-m_0,i}c_{m_0,j}(i+j)Y_0^{i+j-1}.
\end{eqnarray*}
If $l_{m_0}>0$, by comparing the coefficient of $Y_0^{l_{-m_0}+l_{m_0}-1}$, we get $c_{-m_0,l_{-m_0}}c_{m_0,l_{m_0}}(l_{m_0}+l_{-m_0})=0$, which is a contradiction, due to that $c_{-m_0,l_{-m_0}}$, $c_{m_0,l_{m_0}}$ and $l_{m_0}+l_{-m_0}$ are not equal to zero. Therefore, $l_{m_0}=0$, i.e. $d_{m_0,1}=0$, a contradiction.
Therefore, $deg_{L_0}(p_m(L_0, M_0, Y_0))=h_m=0$ or $p_m(L_0, M_0, Y_0)=0$ for any $m\in \mathbb{Z}$, i.e. $p_m(L_0,M_0,Y_0)\in\mathbb{C}\left[M_0,Y_0\right]$.
Consequently, by \eqref{[y,m]},
we have
\begin{align*}
m\Big(\frac{\partial}{\partial Y_0}a_n(L_0, M_0,Y_0)\Big)a_m(L_0,M_0,Y_0)=0.
\end{align*}
Therefore $l_n=0$ for any $n\in\mathbb{Z}$ by Lemma \ref{Lemma1}.
Then
$a_m(L_0,M_0,Y_0)\in\mathbb{C}\left[M_0\right]\setminus \{0\}.$
\end{proof}

By Lemma \ref{Lemma3}, for convenience, we write $a_m(L_0, M_0, Y_0)$ as $a_m(M_0)$ and $p_m(L_0, $ $M_0,Y_0)$ as $p_m(M_0, Y_0)$.
\begin{lemma}\label{Lemma4} $g_m(L_0,M_0,Y_0)=f_{m,1}L_0+f_{m,0}$ with $f_{m,1}\in\mathbb{C}\left[M_0\right]\setminus \{0\}$ and $f_{m, 0}\in \mathbb{C}\left[M_0, Y_0\right]$ for any $m\in \mathbb{Z}$.
\end{lemma}
\begin{proof}
For any $m$, $n\in \mathbb{Z}$, by $(A2)$ and $(A3)$, we get
{\small \begin{eqnarray}
&&\label{eq4}\left[L_m,M_n\right]. 1=(-3m-n)M_{m+n}. 1\\
&=&a_n(M_0)g_m(L_0,M_0,Y_0)-3m \frac{\partial a_n(M_0)}{\partial M_0}a_m(M_0)-g_m(L_0+m, M_0, Y_0) a_n(M_0)\nonumber\\
&=&(-3m-n)a_{m+n}(M_0).\nonumber
\end{eqnarray}}

If there exists $m_0\in \mathbb{Z}^\ast$ such that $g_{m_0}(L_0,M_0,Y_0)=0$ or $deg_{L_0}(g_{m_0}(L_0, M_0, Y_0))=0$. Then by (\ref{eq4}), we get
\begin{eqnarray}\label{eq5}
(-3m_0-n)a_{m_0+n}(M_0)=-3m_0\frac{\partial a_n(M_0)}{\partial M_0} a_{m_0}(M_0).
\end{eqnarray}
For any $n\in \mathbb{Z}$ with $n\neq -3m_0$, we get $deg_{M_0}(a_n(M_0))\geq 1$ due to that $a_{m_0+n}(M_0)\neq 0$. When $n=-3m_0$, by (\ref{eq5}), we have $deg_{M_0}(a_{-3m_0}(M_0))=0$.
Let $n=-4m_0$ and compare the degree of $M_0$ in (\ref{eq5}). The degree of $M_0$ in the left side of (\ref{eq5}) is $0$ and  the degree of $M_0$ in the right side of (\ref{eq5}) is greater than or equal to $1$. We get a contradiction. Then we can assume that $g_m(L_0,M_0,Y_0)=\sum\limits_{i=0}^{t_m}f_{m,i}L^i_0$, where $m\in \mathbb{Z}$,
$f_{m,i}\in\mathbb{C}\left[M_0,Y_0\right]$, $f_{m,t_m}\neq 0$ and $t_m\geq 1$.

Since $deg_{L_0}(a_{m+n}(M_0))=0$, by (\ref{eq4}),
we have
\begin{align*}
deg_{L_0}(a_n(M_0)\sum\limits_{i=0}^{t_m}f_{m,i}L^i_0-\sum\limits_{j=0}^{t_m}f_{m,j} (L_0+n)^ja_n(M_0))=0.
\end{align*}
Then we get that $t_m=1$ for any $m\in \mathbb{Z}$.

Set $g_m(L_0,M_0,Y_0)=f_{m,1}L_0+f_{m,0}$ with $f_{m,i}\in\mathbb{C}\left[M_0,Y_0\right]$ for $i\in \{0,1\}$ and $f_{m,1}\neq 0$. Taking them into (\ref{eq4}), we have
\begin{eqnarray}
\label{eq6}&(-3m-n)a_{m+n}(M_0)=-3m\frac{\partial a_m(M_0)}{\partial M_0}-nf_{m,1}a_n(M_0).
\end{eqnarray}
Comparing the degree of $Y_0$ in (\ref{eq6}), we get $f_{m,1}\in \mathbb{C}[M_0]\setminus\{0\}$.
\end{proof}
\begin{lemma}\label{Lemma5}
$a_m(M_0)=\lambda^mM_0$ and $g_m(L_0, M_0, Y_0)=\lambda^mL_0+f_{m,0}$ for any $m\in \mathbb{Z}$ and some $\lambda \in \mathbb{C}^\ast$, $f_{m,0}\in \mathbb{C}[M_0, Y_0]$.
\end{lemma}
\begin{proof}
By Lemma \ref{Lemma4}, we assume that $g_m(L_0,M_0,Y_0)=f_{m,1}L_0+f_{m,0}$, where $f_{m,1}\in\mathbb{C}\left[M_0\right]\backslash \{0\}$ and $f_{m,0}\in\mathbb{C}\left[M_0, Y_0\right]$.

Let $n=-3m$ for any $m\neq 0$ in (\ref{eq6}). We get
\begin{eqnarray}\label{eq7}
 \frac{\partial a_{-3m}(M_0)}{\partial M_0}a_m(M_0)=f_{m,1}a_{-3m}(M_0).
\end{eqnarray}
Since $f_{m,1}$ and $a_{-3m}(M_0)$ are not equal to zero, $\frac{\partial a_{-3m}(M_0)}{\partial M_0}\neq 0$. Since $deg_{M_0}(a_{-3m}(M_0))=deg_{M_0}( \frac{\partial a_{-3m}(M_0)}{\partial M_0})+1$, we obtain
\begin{eqnarray*}
deg_{M_0}(a_m(M_0))=deg_{M_0}(f_{m,1})+1,\;\; \forall m\in \mathbb{Z}^\ast.
\end{eqnarray*}
Note that $a_0(M_0)=M_0$ and $f_{0,1}=1$. Therefore,
\begin{eqnarray}\label{eq8}
deg_{M_0}(a_m(M_0))=deg_{M_0}(f_{m,1})+1,\;\; \forall m\in \mathbb{Z}.
\end{eqnarray}
By (\ref{eq8}), we have $deg_{M_0}(a_m(M_0))\geq 1$ for any $m\in \mathbb{Z}$. Therefore,
$deg_{M_0}(\frac{\partial a_m(M_0)}{\partial M_0})\geq 0$ for any $m\in \mathbb{Z}$. Consequently, by (\ref{eq8}),
\begin{eqnarray}
\label{eq9} &deg_{M_0}\Big(\frac{\partial a_n(M_0)}{\partial M_0} a_m(M_0)\Big)=deg_{M_0}(f_{m,1}a_n(M_0)), \;\; \forall m, n\in \mathbb{Z},
\end{eqnarray}
according to that $deg_{M_0}(a_n(M_0))=deg_{M_0}(\frac{\partial a_n(M_0)}{\partial M_0})+1$.

By (\ref{eq9}), we set $a_m(M_0)=\sum\limits_{i=0}^{k_m}c_{m,i}M_0^i$ and $f_{m,1}=\sum\limits_{i=0}^{k_m-1}d_{m,i}M^{i}_0$ for any $m\in \mathbb{Z}$, where $k_m\geq 1$, $c_{m, k_m}\neq 0$ and $d_{m,k_m-1}\neq 0$.
If there exist $m_0$ and $n_0\in \mathbb{Z}^\ast$ such that $deg_{M_0}(a_{m_0+n_0}(M_0))<deg_{M_0}(f_{m_0,1}a_{n_0}(M_0))$, i.e. $k_{m_0+n_0}< k_{m_0}+k_{n_0}-1$,
by comparing the coefficients of $M_0^{k_{m_0}+k_{n_0}-1}$ in (\ref{eq6}),
we get
\begin{eqnarray}\label{eq10}
-3m_0k_{n_0}c_{m_0,k_{m_0}}-n_0d_{m_0,k_{m_0}-1}=0.
\end{eqnarray}
Note that $deg_{M_0}\Big(\frac{\partial a_n(M_0)}{\partial M_0}a_m(M_0)\Big)=deg_{M_0}\Big(\frac{\partial a_m(M_0)}{\partial M_0}a_n(M_0)\Big)$ for any $m$, $n\in \mathbb{Z}$. Therefore, by (\ref{eq9}), we also obtain \begin{eqnarray*}
deg_{M_0}(a_{m_0+n_0}(M_0))<deg_{M_0}(f_{n_0,1}a_{m_0}(M_0)).
\end{eqnarray*}
Then similar to that above, we can get
\begin{eqnarray}\label{eq11}
-3n_0k_{m_0}c_{n_0,k_{n_0}}-m_0d_{n_0,k_{n_0}-1}=0.
\end{eqnarray}\label{eq13}
Moreover, by (\ref{eq7}) and $d_{0, k_0-1}=k_0c_{0, k_0}$, we have
\begin{eqnarray}
d_{m, k_m-1}=k_{-3m}c_{m,k_m},\;\; \forall m\in \mathbb{Z}.
\end{eqnarray}
Taking it into (\ref{eq10}) and (\ref{eq11}), we get
\begin{eqnarray}\label{eq14}
3m_0k_{n_0}=-n_0k_{-3m_0},\;\;\;3n_0k_{m_0}=-m_0k_{-3n_0}.
\end{eqnarray}
Note that $k_{n_0}\geq 1$ and $k_{m_0}\geq 1$. We get $m_0n_0< 0$.
On the other hand, we have
\begin{eqnarray}\label{[L,L]}
&\left[L_m,L_n\right]\cdot1=(m-n)(f_{m+n, 1}L_0+f_{m+n,0})\label{LM} \\
&=-3m\frac{\partial f_{n,1}}{\partial M_0}(L_0+m)a_m(M_0)+ mf_{n,1}(f_{m,1}L_0+f_{m,0})-3m\frac{\partial f_{n,0}}{\partial M_0}a_m(M_0)\nonumber\\
&\quad-m\frac{\partial f_{n,0}}{\partial Y_0}p_m(M_0,Y_0)-\frac{m^2}{2}\frac{\partial^2 f_{n,0}}{\partial Y^2_0}a_m(M_0)+3n\frac{\partial f_{m,1}}{\partial M_0}(L_0+n)a_n(M_0)\nonumber\\
&\quad-nf_{m,1}(f_{n,1}L_0+f_{n,0})+3n\frac{\partial f_{m,0}}{\partial M_0}a_n(M_0)
+n\frac{\partial f_{m,0}}{\partial Y_0}p_n(M_0,Y_0)+\frac{n^2}{2}\frac{\partial^2 f_{m,0}}{\partial Y^2_0}a_n(M_0).\nonumber
\end{eqnarray}
Considering the coefficients of $L_0$, we have
\begin{eqnarray}
\label{eq12}&&-3m\frac{\partial f_{n,1}}{\partial M_0}a_m(M_0)+mf_{n,1}f_{m,1}+3n\frac{\partial f_{m,1}}{\partial M_0}a_n(M_0)+mf_{n,1}f_{m,1}-nf_{m,1}f_{n,1}\\
&&=(m-n)f_{m+n,1}.\nonumber
\end{eqnarray}
Let $m=m_0$ and $n=n_0$ in (\ref{eq12}). Since $k_{m_0+n_0}< k_{m_0}+k_{n_0}-1$, comparing the coefficients of $M_0^{k_{n_0}+k_{m_0}-2}$, we get
\begin{eqnarray}
&&\label{eq15}-3m_0(k_{n_0}-1)d_{n_0,k_{n_0}-1}c_{m_0, k_{m_0}}+m_0d_{n_0,k_{n_0}-1}d_{m_0,k_{m_0}-1}\\
&&+3n_0(k_{m_0}-1)d_{m_0,k_{m_0}-1}c_{n_0,k_{n_0}}-n_0d_{m_0,k_{m_0}-1}d_{n_0,k_{n_0}-1}=0\nonumber.
\end{eqnarray}
Using (\ref{eq13}) and (\ref{eq14}), (\ref{eq15}) can be reduced to
\begin{eqnarray*}
m_0k_{-3n_0}-n_0k_{-3m_0}=0,
\end{eqnarray*}
which means that $m_0n_0>0$. Therefore, we get a contradiction. Consequently, we have
\begin{eqnarray}
\label{eq16}deg_{M_0}(a_{m+n}(M_0))=deg_{M_0}(\frac{\partial a_n(M_0)}{\partial M_0}a_m(M_0)), \;\; \forall m, n\in \mathbb{Z}.
\end{eqnarray}
Note that $deg_{M_0}(a_{0}(M_0))=1$. Therefore, by (\ref{eq16}), we get
\begin{eqnarray*}
deg_{M_0}(\frac{\partial a_{-m}(M_0)}{\partial M_0}a_m(M_0))=1, \forall m\in \mathbb{Z}.
\end{eqnarray*}
Thus, $deg_{M_0}(a_m(M_0))=1$ for any $m\in \mathbb{Z}$ due to that $deg_{M_0}(a_m(M_0))\geq 1$.
Therefore, $deg_{M_0}(f_{m,1})=0$ for any $m\in \mathbb{Z}$, i.e. $f_{m,1}\in\mathbb{C}^\ast$.

We assume that
\begin{align*}
a_m(M_0)=c_{m,0}+c_{m,1}M_0,\forall m\in\mathbb{Z}, c_{m,0}, c_{m,1}\in \mathbb{C}.
\end{align*}
Taking it into (\ref{eq6}), we have
\begin{align*}
&(-3m-n)(c_{m+n,0}+c_{m+n,1}M_0)\\
&=-3mc_{n,1}c_{m,1}M_0-3mc_{n,1}c_{m,0}-nf_{m,1}c_{n,1}M_0-nf_{m,1}c_{n,0}.
\end{align*}
Considering the coefficients of $M_0$, we have
\begin{align*}
(-3m-n)c_{m+n,1}=-3mc_{n,1}c_{m,1}-nf_{m,1}c_{n,1}.
\end{align*}
Taking $n=-3m$, we have $c_{m,1}=f_{m,1}$. Then if $n\neq -3m$, then $c_{m+n,1}=c_{n,1}c_{m,1}$.
Assume that $f_{1,1}=c_{1,1}=\lambda\in \mathbb{C}^\ast$. Then we have $c_{m,1}=f_{m,1}=\lambda^m$ for any $m\in \mathbb{Z}$. Similarly, if $n\neq -3m$, we get $c_{m+n, 0}=c_{n,1}c_{m,0}$. Let $m=0$. We get $c_{m,0}=0$ for any $m\in \mathbb{Z}$ due to that $c_{0,0}=0$.
Therefore, we have
$a_m(M_0)=\lambda^mM_0$ and $f_{m,1}=\lambda^m$ for any $m\in \mathbb{Z}$ and some $\lambda \in \mathbb{C}^\ast$.
\end{proof}
\begin{lemma}\label{Y1}
For $p_m(M_0,Y_0)$, there are two cases.

(a) \begin{eqnarray*}
p_n(M_0,Y_0)=\frac{m_0-n}{m_0}\lambda^nY_0+\alpha_n(M_0),
\end{eqnarray*}
for some $m_0\in \mathbb{Z}^\ast$, where $\alpha_n(M_0)\in \mathbb{C}[M_0]$, $n\in \mathbb{Z}$ and $\alpha_0(M_0)=0$.

(b) \begin{eqnarray*}
p_n(M_0,Y_0)=n \lambda^n \sum\limits_{i=1}^{q}T_iY_0^i+\lambda^nY_0+T_{n,0},
\end{eqnarray*}
where $q\in \mathbb{N}$, $n\in\mathbb{Z}$, $ T_i$, $T_{n,0}\in\mathbb{C}[M_0]$ and $T_{0,0}=0$.

\end{lemma}
\begin{proof}
For any $m$, $n\in \mathbb{Z}$, we have
\begin{eqnarray}
\label{[Y,Y]}\left[Y_m,Y_n\right]. 1&=&(m-n)M_{m+n}. 1 \\
&=&Y_m. (Y_n. 1)-Y_n. (Y_m. 1)\nonumber\\
&=&m\frac{\partial p_n(M_0,Y_0)}{\partial Y_0}a_m(M_0)-n\frac{\partial p_m(M_0,Y_0)}{\partial Y_0}a_n(M_0)\nonumber\\
&=&(m-n)a_{m+n}(M_0).\nonumber
\end{eqnarray}
Note that $a_m(M_0)=\lambda^mM_0$ for any $m\in \mathbb{Z}$.

If there exists some $m_0\in \mathbb{Z}^\ast$ such that $p_{m_0}(M_0, Y_0)=0$ or $deg_{Y_0}(p_{m_0}(M_0,Y_0))=0$, setting $m=m_0$ in (\ref{[Y,Y]}), we get
\begin{eqnarray}
p_n(M_0,Y_0)=\frac{m_0-n}{m_0}\lambda^nY_0+\alpha_n(M_0),\;\;
\end{eqnarray}
for any $n\in \mathbb{Z}$ with $n\neq m_0$ and some $\alpha_n(M_0)\in \mathbb{C}[M_0]$.
Note that $\alpha_0(M_0)=0$ by $p_0(M_0,Y_0)=Y_0$. This is Case $(a)$.

Next, by Case (a), we can assume that $deg_{Y_0}(p_m(M_0,Y_0))\geq 1$ for any $m\in \mathbb{Z}$. Set
$p_m(M_0,Y_0)= \sum\limits_{i=0}^{k_m}e_{m,i}Y^i_0$,
where $e_{m,i}\in\mathbb{C}\left[M_0\right]$, $k_{m, k_m}\neq 0$ and $k_m\geq 1$.

If there exists $m_0\in \mathbb{Z}^\ast$ such that $k_{m_0}=q\geq 2$, then we get $k_n=q$ for any $n\in \mathbb{Z}^\ast$ by comparing the degrees of $Y_0$ in (\ref{[Y,Y]}).
Therefore, the last case is $k_n=1$ for any $n\in \mathbb{Z}$. In both two cases, we can set
$p_m(M_0,Y_0)= \sum\limits_{i=0}^{q}e_{m,i}Y^i_0$, where $q\geq 1$, $e_{m,i}\in\mathbb{C}\left[M_0\right]$ and $k_{m, q}\neq 0$. Taking them into (\ref{[Y,Y]}), we have
\begin{eqnarray}
&&\label{eqq1}m\lambda^m\sum\limits_{i=2}^{q}e_{n,i}Y^{i-1}_0-n\lambda^n\sum\limits_{i=2}^{q}e_{m,i}Y^{i-1}_0=0,\\
&&\label{eqq2}m\lambda^me_{n,1}-n\lambda^ne_{m,1}=(m-n)\lambda^{m+n}.
\end{eqnarray}
If $q=1$, (\ref{eqq1}) naturally holds. If $q\geq 2$, by (\ref{eqq1}), we get
\begin{eqnarray*}
m\lambda^m e_{n,i}=n\lambda^ne_{m,i}, \;\;\;\text{ $i=2$, $3$, $\cdots$, $q$,}
\end{eqnarray*}
for any $m$, $n\in \mathbb{Z}$. Then one obtains $e_{m,i}=m\lambda^mT_i$, where $T_i\in\mathbb{C}\left[M_0\right]$ and $i\geq 2$.
By (\ref{eqq2}), we have $e_{m,1}=m\lambda^mT_1+\lambda^m$, where $T_1\in\mathbb{C}[M_0].$
Then we have
\begin{align}\label{pm}
p_m(M_0,Y_0)=m \lambda^m \sum\limits_{i=1}^{q}T_iY_0^i+\lambda^mY_0+T_{m,0},
\end{align}
where $m\in\mathbb{Z}^*,\ \lambda\in\mathbb{C}^*,\ T_i,\ T_{m,0}\in\mathbb{C}[M_0].$ Note that $p_0(M_0,Y_0)=Y_0$. Therefore $T_{0,0}=0$. This is Case $(b)$.
\end{proof}
\begin{lemma}\label{Lemma6}
Case $(a)$ in Lemma \ref{Y1} does not hold for $B$ to be a module of $\mathfrak{tsv}$ with the actions as $(A1)$-$(A3)$.
\end{lemma}
\begin{proof}
By $(A1)$ and $(A3)$, we have
\begin{eqnarray}
\left[L_m,Y_n\right]. 1&=&(-m-n)p_{m+n}(M_0,Y_0)\nonumber\\
&=&-m\frac{\partial p_n(M_0,Y_0)}{\partial Y_0}p_m(M_0,Y_0)-\frac{m^2}{2}\frac{\partial^2p_n(M_0,Y_0)}{\partial Y^2_0}a_m(M_0,Y_0)\nonumber\\
\label{LY}&&\hspace{0.4cm}-3m\frac{\partial p_n(M_0, Y_0)}{\partial M_0}a_m(M_0)-nf_{m,1}p_n(M_0,Y_0)-n\frac{\partial f_{m,0}}{\partial Y_0}a_n(M_0) .
\end{eqnarray}
Let $n=m_0$ in (\ref{LY}). Since $p_n(M_0,Y_0)=\frac{m_0-n}{m_0}\lambda^nY_0+\alpha_n(M_0)$ for any $n\in \mathbb{Z}$ in Case $(a)$, we get
\begin{eqnarray}
&&\label{eqq3}-3m\lambda^m\frac{\partial \alpha_{m_0}(M_0)}{\partial M_0}-m_0\lambda^m \alpha_{m_0}(M_0)-m_0\lambda^{m_0}\frac{\partial f_{m,0}}{\partial Y_0}\\
&=&-(m_0+m)(-\frac{m}{m_0}\lambda^{m_0+m}Y_0+\alpha_{m_0+m}(M_0)).\nonumber
\end{eqnarray}
By comparing the coefficients of $Y_0$ in (\ref{eqq3}), one obtains
\begin{eqnarray}
&& f_{m,0}=-\frac{m(m+m_0)}{2m_0^2}\lambda^mY_0^2+\beta_m(M_0)Y_0+\gamma_m(M_0),
\end{eqnarray}
for any $m\in \mathbb{Z}$, where $\beta_m(M_0)$, $\gamma_m(M_0)\in \mathbb{C}[M_0]$. Taking them into (\ref{[L,L]}) and comparing the coefficients of $Y_0^2$, we have
\begin{eqnarray*}
&&-m\frac{m(m_0+m)}{2m_0^2}+m(m_0-m)\frac{n(m_0+n)}{2m_0^2}+n\frac{n(m_0+n)}{2m_0^2}-n(m_0-n)\frac{m(m_0+m)}{2}\\
&&=-(m-n)\frac{(m+n)(m_0+m+n)}{2m_0^2},
\end{eqnarray*}
which is equivalent to $2m_0(mn^2-m^2n)=mn^2-m^2n$ for any $m$, $n\in \mathbb{Z}$. Therefore $m_0=\frac{1}{2} \notin \mathbb{Z}$. We get a contradiction. Then this lemma holds.
\end{proof}
\begin{lemma} $p_m(L_0,M_0,Y_0)=m \lambda^m \sum\limits_{i=0}^{q}T_iY_0^i+\lambda^mY_0+am^2\lambda^mM_0+b\lambda^m,\forall m\in\mathbb{Z}^\ast,$ where $q\geq 1$, $a$, $b\in\mathbb{C}$, $T_i\in\mathbb{C}[M_0].$
\end{lemma}
\begin{proof}
By Lemmas \ref{Lemma5}, \ref{Y1} and \ref{Lemma6},
we assume that
\begin{eqnarray*}
&&g_m(L_0,M_0,Y_0)=\lambda^mL_0+\sum\limits_{i=0}^{t_m}r_{m,i}Y_0^i,\\
&&p_m(M_0,Y_0)=m \lambda^m \sum\limits_{i=1}^{q}T_iY_0^i+\lambda^mY_0+T_{m,0},
\end{eqnarray*}
where $r_{m,i}\in\mathbb{C}[M_0]$, $q\geq 1$, $n\in\mathbb{Z}$, $ T_i$, $T_{n,0}\in\mathbb{C}[M_0]$ and $T_{0,0}=0$. Then substituting them into $\eqref{LY}$, we have
$$
\begin{aligned}
&[L_m,Y_n]\cdot1=-(m+n)\left((m+n)\lambda^{m+n}\sum\limits_{i=1}^{q}T_iY_0^i+\lambda^{m+n}Y_0+T_{m+n,0}\right)\\
&=-m\left(n\lambda^n\sum\limits_{i=1}^{q}iT_iY_0^{i-1}+\lambda^n\right)
\left(m\lambda^m\sum\limits_{i=1}^{q}T_iY_0^i+\lambda^mY_0+T_{m,0}\right)\\
&-\frac{m^2}{2}\lambda^mM_0\left(n\lambda^n\sum\limits_{i=1}^{q}i(i-1)T_iY_0^{i-2}\right)
-3m\lambda^mM_0\left(n\lambda^n\sum\limits_{i=1}^{q}\frac{\partial T_i}{\partial M_0}Y_0^i+\frac{\partial T_{n,0}}{\partial M_0}\right)\\
&-n\lambda^m\left(n\lambda^n\sum\limits_{i=1}^{q}T_iY_0^i+\lambda^nY_0+T_{n,0}\right)
-n\lambda^nM_0\sum\limits_{i=1}^{t_m}ir_{m,i}Y_0^{i-1}\\
\end{aligned}
$$

By  observing the constant terms in terms of $Y_0$, we have
\begin{eqnarray}\label{3}
 -(m+n)T_{m+n,0}&=&m(n\lambda^nT_1+\lambda^n)T_{m,0}-m^2n\lambda^nT_2\lambda^mM_0\\
  &&-3m\frac{\partial T_{n,0}}{\partial M_0}\lambda^mM_0-n\lambda^mT_{n,0}-nr_{m,1}\lambda^nM_0.\nonumber
\end{eqnarray}
Note that if $q=1$, $T_2=0$.
Setting $n=\pm1$ in (\ref{3}), we have
\begin{eqnarray}\label{*1}
-(m+1)T_{m+1,0}&=&-m(\lambda T_1+\lambda)T_{m,0}-m^2T_2\lambda^{m+1}M_0\\
&&-3m\frac{\partial T_{1,0}}{\partial M_0}\lambda^mM_0-\lambda^mT_{1,0}-r_{m,1}\lambda M_0\nonumber,
\end{eqnarray}
and
\begin{eqnarray}\label{*2}
-(m-1)T_{m-1,0}&=&-m\Big(-\frac{1}{\lambda}T_1+\frac{1}{\lambda}\Big)T_{m,0}+m^2T_2\lambda^{m-1}M_0\\
&&-3m\frac{\partial T_{-1,0}}{\partial M_0}\lambda^mM_0+\lambda^mT_{-1,0}+r_{m,1}\frac{1}{\lambda}M_0.\nonumber
\end{eqnarray}
Using (\ref{*1}) and (\ref{*2}), we get
\begin{align}\label{1}
 \begin{split}
  &2\lambda mT_{m,0}+3m\frac{\partial T_{1,0}}{\partial M_0}\lambda^mM_0+3m\frac{\partial T_{-1,0}}{\partial M_0}\lambda^{m+2}M_0+\lambda^mT_{1,0}-\lambda^{m+2}T_{-1,0}\\
  &=(m+1)T_{m+1,0}+(m-1)T_{m-1,0}\lambda^2.
 \end{split}
\end{align}
Replacing $m$ by $-m$ in (\ref{1}), one obtains
\small{\begin{align}\label{2}
 \begin{split}
   &-2\lambda^{2m+1} mT_{-m,0}-3m\frac{\partial T_{1,0}}{\partial M_0}\lambda^mM_0-3m\frac{\partial T_{-1,0}}{\partial M_0}\lambda^{m+2}M_0+\lambda^mT_{1,0}-\lambda^{m+2}T_{-1,0}\\
  &=(-m+1)\lambda^{2m}T_{-m+1,0}+(-m-1)T_{-m-1,0}\lambda^{2m+2}.
  \end{split}
\end{align}}
Adding $\eqref{1}$ and $\eqref{2}$ up, we have
\small{\begin{align*}
&2\lambda mT_{m,0}-2\lambda^{2m+1} mT_{-m,0}+2\lambda^mT_{1,0}-2\lambda^{m+2}T_{-1,0}\\
&=(m+1)T_{m+1,0}+(m-1)T_{m-1,0}\lambda^2+(-m+1)\lambda^{2m}T_{-m+1,0}+(-m-1)T_{-m-1,0}\lambda^{2m+2}.
\end{align*}}
Set $b_t=t\lambda^{m+1-t}T_{t_,0}$. Then we get
\begin{align*}
2b_m+2b_{-m}+2b_1+2b_{-1}=b_{m+1}+b_{m-1}+b_{-m+1}+b_{-m-1}.
\end{align*}
Then
\begin{align*}
[(b_{m+1}-b_m)+(b_{-(m+1)}-b_{-m})]-[(b_m-b_{m-1})+(b_{-m}-b_{-(m-1)})]=2b_1+2b_{-1}.
\end{align*}
Therefore, we obtain
\begin{align*}
(b_m-b_{m-1})+(b_{-m}-b_{-(m-1)})=2(m-1)(b_1+b_{-1})+(b_1+b_{-1})=(2m-1)(b_1+b_{-1}).
\end{align*}
Thus
\begin{align*}
b_m+b_{-m}=m^2(b_1+b_{-1})+2b_0=m^2(b_1+b_{-1}).
\end{align*}
That is
\begin{eqnarray}\label{*3}
\lambda T_{m,0}-\lambda^{2m+1}T_{-m,0}=m(\lambda^mT_{1,0}-\lambda^{m+2}T_{-1,0}).
\end{eqnarray}
Subtracting $\eqref{2}$ from $\eqref{1}$, we have
\begin{align*}
&2\lambda mT_{m,0}+2\lambda^{2m+1}mT_{-m,0}+6m\frac{\partial T_{1,0}}{\partial M_0}\lambda^mM_0+6m\frac{\partial T_{-1,0}}{\partial M_0}\lambda^{m+2}M_0\\
&=(m+1)T_{m+1,0}-(-m+1)T_{-m+1,0}\lambda^{2m}+(m-1)T_{m-1,0}\lambda^2-(-m-1)T_{-m-1,0}\lambda^{2m+2}.
\end{align*}
Setting $a_t=\lambda^{m+1-t}T_{t,0}$, we have
\begin{align*}
&2ma_m+2ma_{-m}+6mM_0\frac{\partial(a_1+a_{-1})}{\partial M_0}\\
&=(m+1)a_{m+1}-(-m+1)a_{-m+1}+(m-1)a_{m-1}-(-m-1)a_{-m-1}.
\end{align*}
Then we have
\begin{align*}
&[(m+1)(a_{m+1}+a_{-(m+1)})-m(a_m+a_{-m}]-[m(a_m+a_{-m})-(m-1)(a_{m-1}+a_{-m+1})]\\
&=6mM_0\frac{\partial(a_1+a_{-1})}{\partial M_0}.
\end{align*}
Therefore, one gets
\begin{align*}
m(a_m+a_{-m})=\left[\frac{m(m+1)(2m+1)}{12}-\frac{m(m+1)}{4}\right]\cdot6M_0\frac{\partial(a_1+a_{-1})}{\partial M_0}+m(a_{-1}+a_1).
\end{align*}
Thus
\begin{align*}
a_m+a_{-m}=(m^2-1)M_0\frac{\partial(a_1+a_{-1})}{\partial M_0}+(a_{-1}+a_1),\;\;(m\neq0).
\end{align*}
That is
\begin{eqnarray}\label{*4}
&&\lambda T_{m,0}+\lambda^{2m+1}T_{-m,0}=(m^2-1)M_0\lambda^m\frac{\partial(T_{1,0}+\lambda^2T_{-1,0})}{\partial M_0}+\lambda^mT_{1,0}+\lambda^{m+2}T_{-1,0}.
\end{eqnarray}
Setting  $m=1$ in $\eqref{1}$ and $A=\frac{\partial (T_{1,0}+\lambda^2 T_{-1,0})}{\partial M_0}M_0$,  we have
\begin{align}\label{*5}
T_{2,0}=\frac{3}{2}\lambda T_{1,0}+\frac{3}{2}\lambda A-\frac{1}{2}\lambda ^3T_{-1,0}.
\end{align}
By (\ref{*3}), we get
\begin{align}\label{*6}
T_{-2,0}=\frac{3}{2}\lambda ^{-1}T_{-1,0}+\frac{3}{2}\lambda^{-3} A-\frac{1}{2}\lambda ^{-3}T_{1,0}.
\end{align}
Setting $m=2$ in $\eqref{1}$, we have
\begin{align}\label{*7}
T_{3,0}=2\lambda^2 T_{1,0}+4\lambda^2 A-\lambda ^4T_{-1,0}.
\end{align}
Replacing $n$ by $-n$ in $\eqref{3}$, we have
\begin{align}\label{4}
   \begin{split}
   &-m(-n\lambda^{-n}T_1+\lambda^{-n})T_{m,0}+m^2n\lambda^{-n}T_2\lambda^mM_0
   -3m\frac{\partial T_{-n,0}}{\partial M_0}\lambda^mM_0\\
   &+n\lambda^mT_{-n,0}+nr_{m,1}\lambda^{-n}M_0
   = -(m-n)T_{m-n,0}.
   \end{split}
\end{align}
By $\eqref{3}$ and $\eqref{4}$, we have
\begin{align*}
&-2m\lambda^nT_{m,0}-3m\lambda^m\frac{\partial T_{n,0}}{\partial M_0}M_0-3m\lambda^{m+2n}\frac{\partial T_{-n,0}}{\partial M_0}M_0-n\lambda^mT_{n,0}+n\lambda^{m+2n}T_{-n,0}\\
&=-(m+n)T_{m+n,0} -(m-n)\lambda^{2n}T_{m-n,0}.
\end{align*}
Setting $n=2$ and $m=1$ in the above equality, we have
\begin{align*}
-2\lambda^2T_{1,0}-3\lambda\frac{\partial T_{2,0}}{\partial M_0}M_0-3\lambda^5\frac{\partial T_{-2,0}}{\partial M_0}M_0-2\lambda T_{2,0}+2\lambda^5T_{-2,0}=-3T_{3,0} +\lambda^4T_{-1,0}.
\end{align*}
Taking (\ref{*5}), (\ref{*6}) and (\ref{*7}) into the above equality, we get $M_0\frac{\partial A}{\partial M_0}=A$. Thus $A=cM_0$, where $c\in\mathbb{C}$.
Therefore,
\begin{align}\label{*8}
T_{1,0}+\lambda^2T_{-1,0}=c_2M_0+c_1,
\end{align}
where $c_1,c_2\in\mathbb{C}$.
Using (\ref{*3}), (\ref{*4}) and (\ref{*8}), we obtain
\begin{align*}
T_{m,0}=\lambda^{m-1}\Big(mT_{1,0}+\frac{c_2}{2}(m^2-m)M_0+\frac{c_1}{2}(1-m)\Big), \;\;m\neq0,
\end{align*}
where $T_{1,0}\in\mathbb{C}[M_0]$.

Set
\begin{align*}
T_0=\frac{T_{1,0}}{\lambda}-\frac{c_2M_0}{2\lambda}-\frac{c_1}{2\lambda},a=\frac{c_2}{2\lambda},b=\frac{c_1}{2\lambda}.
\end{align*}
Then
\begin{align*}
T_{m,0}=m\lambda^mT_0+am^2\lambda^mM_0+b\lambda^m,\;\;m\neq0.
\end{align*}
Then we get
\begin{align}\label{pm2}
 p_m(L_0,M_0,Y_0)=m \lambda^m \sum\limits_{i=0}^{q}T_iY_0^i+\lambda^mY_0+am^2\lambda^mM_0+b\lambda^m,\;\;\forall m\in\mathbb{Z}^*,\;\;q\geq 1, \; a, b \in \mathbb{C}.
\end{align}
\end{proof}
\begin{lemma}\label{Lemma10}
\begin{eqnarray}
g_m(L_0,M_0,Y_0)&=&\lambda^mL_0+\frac{1}{M_0}\bigg[m\lambda^m\sum\limits_{i=0}^{q}
\left(-T_i+\frac{3}{i+1}T_i-\frac{3M_0}{i+1}\frac{\partial T_i}{\partial M_0}\right)Y_0^{i+1}\nonumber\\
&&-\frac{m^2}{2}\lambda^m M_0 \sum\limits_{i=0}^{q}iT_iY_0^{i-1}-\frac{m^2}{2}\lambda^m\left(\sum\limits_{i=0}^{q}T_iY_0^i\right)^2+3am^2\lambda^mM_0Y_0\nonumber\\
\label{*10}&&-(am^2M_0+b)m\lambda^m\sum\limits_{i=0}^{q}T_iY_0^i+\lambda^mR_m\bigg], \end{eqnarray}
where $a,$ $b\in\mathbb{C}$,  $\lambda\in\mathbb{C}^*$, $q\in \mathbb{N}$, $T_i$, $R_m\in M_0\mathbb{C}[M_0]$ for $i\in \{0, 1, \cdots, q\}$ and $m\in\mathbb{Z}$, and $R_0=0$.
\end{lemma}
\begin{proof}
Obviously, $\eqref{LY}$ holds when $m=0$ or $n=0$. So we can assume that $m\neq 0$ and $n\neq 0$.
Set $w= \sum\limits_{i=0}^{q}T_iY_0^i$. Then taking $\eqref{pm2}$ into $\eqref{LY}$, we obtain
\begin{eqnarray*}
[L_m,Y_n].1&=& -(m+n)((m+n)\lambda^{m+n}w+\lambda^{m+n}Y_0+a(m+n)^2\lambda^{m+n}M_0+b\lambda^{m+n})\\
&=&-m\Big(n\lambda^n\frac{\partial w}{\partial Y_0}+\lambda^n\Big)(m\lambda^mw+\lambda^mY_0+am^2\lambda^mM_0+b\lambda^m)\\
&&-\frac{m^2}{2}\Big(n\lambda^n\frac{\partial^2w}{\partial Y_0^2}\Big)\lambda^mM_0
-3m\Big(n\lambda^n\frac{\partial w}{\partial M_0}+an^2\lambda^n\Big)\lambda^mM_0\\
&&-n\lambda^m(n\lambda^nw+\lambda^nY_0+an^2\lambda^nM_0+b\lambda^n)
-n\frac{\partial f_{m,0}}{\partial Y_0}\lambda^nM_0.
\end{eqnarray*}
Obviously, if $m+n=0$, the above equality is equivalent to $\eqref{LY}$.
Setting $w_0=\sum\limits_{i=0}^{q}T'_iY_0^i$, which is the constant term of $w$ with respect to $M_0$, and $T'_{i}$ is the constant term of $T_i$ with respect to $M_0$, and considering the constant term of the equation, we have
\begin{align*}
-(m+n)((m+n)w_0+Y_0+b)=-m(n\frac{\partial w_0}{\partial Y_0}+1)(mw_0+Y_0+b)-n(nw_0+Y_0+b).
\end{align*}
By calculating, we have
\begin{align*}
2w_0=m\frac{\partial w_0}{\partial Y_0}w_0+\frac{\partial w_0}{\partial Y_0}Y_0+b\frac{\partial w_0}{\partial Y_0}.
\end{align*}
Since the above equality holds for any  $m\in \mathbb{Z}^\ast$, we obtain that $w_0=0$, that is $T'_i=0$ for $i\in \{0, 1, \cdots, q\}$. Therefore, $w$ can be divisible by $M_0$.

Setting $n=1$, we have
\begin{align*}
&\frac{\partial f_{m,0}}{\partial Y_0} = \frac{1}{M_0}\bigg[2m\lambda^mw-m\lambda^m\frac{\partial w}{\partial Y_0}(Y_0+b)-\frac{m^2}{2}\lambda^m\frac{\partial^2w}{\partial Y_0^2}M_0-3m\lambda^m\frac{\partial w}{\partial M_0}M_0\\
&\quad\quad\quad\quad-m^2\lambda^m\frac{\partial w}{\partial Y_0}w+3am^2\lambda^mM_0-am^3\lambda^m\frac{\partial w}{\partial Y_0}M_0\bigg].
\end{align*}
Then we get
{\small \begin{eqnarray}
f_{m,0}& =& \frac{1}{M_0}\bigg[-m\lambda^mw(Y_0+b)-\frac{m^2}{2}\lambda^m\frac{\partial w}{\partial Y_0}M_0-m^2\lambda^m\frac{w^2}{2}+3m\lambda^m\sum\limits_{i=0}^{q}\frac{1}{i+1}T_iY_0^{i+1}\nonumber\\
&&\label{*11}\quad -3m\lambda^m\sum\limits_{i=0}^{q}\frac{M_0}{i+1}\frac{\partial T_i}{\partial M_0}Y_0^{i+1}+3am^2\lambda^mM_0Y_0-am^3\lambda^mwM_0+r_m\bigg],
\end{eqnarray}}
where $r_m\in M_0\mathbb{C}[M_0]$, $m\in \mathbb{Z}^\ast$.

Set $R_m=\frac{r_m}{\lambda^m}$. Then taking (\ref{*11}) in to
$g_m(L_0,M_0,Y_0)=\lambda^mL_0+f_{m,0}$, we get (\ref{*10}).  Note that $R_0=0$, due to $g_0(L_0, M_0, Y_0)=L_0$. Then the proof is finished.
\end{proof}

{\bf The proof of Theorem \ref{theorem of isomorphism}}:

By Lemmas \ref{Lemma1}-\ref{Lemma10}, we have
$$
\begin{aligned}
&g_m(L_0,M_0,Y_0)=\lambda^mL_0+\frac{1}{M_0}\Bigg[m\lambda^m\sum\limits_{i=0}^{q}
\left(-T_i+\frac{3}{i+1}T_i-\frac{3M_0}{i+1}\frac{\partial T_i}{\partial M_0}\right)Y_0^{i+1}\\
&\quad\quad\quad\quad-\frac{m^2}{2}\lambda^m M_0 \sum\limits_{i=0}^{q}iT_iY_0^{i-1}
-\frac{m^2}{2}\lambda^m\bigg(\sum\limits_{i=0}^{q}T_iY_0^i\bigg)^2+3am^2\lambda^mM_0Y_0\\
&\quad\quad\quad\quad-(am^2M_0+b)m\lambda^m\sum\limits_{i=0}^{q}T_iY_0^i+\lambda^mR_m\Bigg],\\
&a_m(L_0,M_0,Y_0)=\lambda^mM_0,\\
&p_m(L_0,M_0,Y_0)=m \lambda^m \sum\limits_{i=0}^{q}T_iY_0^i+\lambda^mY_0+am^2\lambda^mM_0+(1-\delta_{m, 0})b\lambda^m,
\end{aligned}
$$
where $a,b\in\mathbb{C},\ m\in\mathbb{Z},\ \lambda\in\mathbb{C}^*$, $T_i$, $R_m\in M_0\mathbb{C}[M_0]$, $R_0=0$ and $q\in\mathbb{N}$.

Then by Lemma \ref{lemma3.1} and Proposition \ref{lie}, $B$ is a module of $\mathfrak{tsv}$ determined by $g_m(L_0, M_0, Y_0)$, $a_m(L_0, M_0, Y_0)$ and $p_m(L_0, M_0, Y_0)$ as above.
Define a linear map $\varphi: B\rightarrow \Phi(\lambda, a, b, q, \{T_i\}, \{R_j\} )=\mathbb{C}[s, t, v]$ as follows
\begin{eqnarray}
\varphi(L_0^iM_0^jY_0^k)=s^it^jv^k, \;\; \forall i, j, k\in \mathbb{Z}_+.
\end{eqnarray}
Then it is easy to see that $\varphi$ is an isomorphism of $\mathfrak{tsv}$-modules. Therefore,
$B\cong \Phi(\lambda, a, b, q, \{T_i\}, \{R_j\})$. The proof is finished.
 \qed

\end{document}